\documentclass[twoside]{article}
\usepackage{lipsum}
\usepackage[margin=4cm]{geometry}
\usepackage[affil-it]{authblk}
\usepackage{amsmath, amsthm, amssymb}
\usepackage{mathrsfs}
\usepackage{fancyhdr}
\usepackage{enumerate}

\newcommand{\Z}{\mathbb Z}

\newcommand{\C}{\mathbb C}

\newcommand{\F}{\mathbb F}

\newtheorem{thm}{Theorem}[section]
\newtheorem{lem}[thm]{Lemma}
\newtheorem{prop}[thm]{Proposition}
\newtheorem{rem}[thm]{Remark}
\newtheorem{ex}{Example}[section]

\DeclareMathOperator{\End}{End}
\DeclareMathOperator{\Sym}{Sym}

\DeclareMathOperator{\Ad}{Ad}
\DeclareMathOperator{\GL}{GL}
\DeclareMathOperator{\GO}{GO}
\DeclareMathOperator{\GSp}{GSp}
\DeclareMathOperator{\PGL}{PGL}
\DeclareMathOperator{\PSL}{PSL}
\DeclareMathOperator{\PSp}{PSp}
\DeclareMathOperator{\SL}{SL}
\DeclareMathOperator{\Sp}{Sp}
\DeclareMathOperator{\SU}{SU}

\DeclareMathOperator{\Ind}{Ind}
\DeclareMathOperator{\Res}{Res}
\DeclareMathOperator{\Hom}{Hom}
\DeclareMathOperator{\Gal}{Gal}
\DeclareMathOperator{\Frob}{Frob}
\DeclareMathOperator{\tr}{tr}

\fancypagestyle{mypagestyle}{%
  \fancyhf{}
  \fancyhead[OC]{Irreducible adjoint representations in prime power dimensions, with an application}
  \fancyhead[EC]{\sc{Liubomir Chiriac}}
  \fancyfoot[C]{\thepage}%
}
\begin{document}
\title{Irreducible adjoint representations in prime power dimensions, with an application}
\author{\sc{Liubomir Chiriac}}
\date{}
\maketitle
\begin{abstract}
We construct an infinite family of representations of finite groups with an irreducible adjoint action and we give an application to the question of lacunary of Frobenius traces in Galois representations.

\smallskip

2010 \emph{Mathematics Subject Classification}. 11F80, 11R45, 20C15.
\end{abstract}

\section{Introduction.}  Let $n\geq 2$ be an integer. Consider the representation $$\Ad:\GL_n(\C)\to \GL_{n^2-1}(\C)$$ obtained by composing the natural projection of $\GL_n(\C)$ onto $\PGL_n(\C)$ with the $(n^2-1)$-dimensional adjoint representation of $\PGL_n(\C)$. Henceforth, given a representation $\rho$, we shall to refer to $\Ad(\rho)$ as the adjoint of $\rho$. The main question we address is whether for every $n$ there exists a finite group $G$ and an $n$-dimensional irreducible $\C$-representation $(\rho,V)$ of $G$ such that $\Ad(\rho)$ is irreducible. We will then apply our conclusion to Galois representations with finite image.

\smallskip

For $n\leq 4$ we find all the groups $G$ with the required property using Blichfeldt's classifications (cf. \cite{Blich}) of the finite subgroups of $\PGL_n(\C)$ for $n\leq 4$. In Proposition 3.1 we show that if $n=2$ then $\Ad(\rho)=\Sym^2(\rho)\otimes \det(\rho)^{-1}$ is irreducible unless $\rho$ is of cyclic or dihedral type. In Proposition 3.2 we prove that for $n=3$, $\Ad(\rho)$ is irreducible in precisely four cases; two of which yield simple groups and the remaining two yield solvable groups. Similarly, Proposition 3.3 shows that for $n=4$ there are two instances when $\Ad(\rho)$ is irreducible. Note that for $n\geq 3$, $\Ad(\rho)$ will be reducible if $\rho$ were essentially self-dual. On the other hand, many examples of Galois representations coming from arithmetic geometry tend to be essentially self-dual.

\smallskip

If we restrict our attention to quasisimple groups $G$ (i.e., perfect central extensions of simple groups), then there is a complete answer in a paper of Magaard, Malle and Tiep \cite{MMT} to whether there are irreducible representations $\rho$ of $G$ with $\Ad(\rho)$ irreducible. As it turns out, the only infinite family of examples are the Weil representations for $\SU_{n}(\F_2)$ and $\Sp_{2n}(\F_3)$ of degrees $(2^n-(-1)^n)/3$ and $(3^n-(-1)^n)/2$, respectively (see also \cite{Mal}). This will not provide examples in the generic prime power case, forcing us to consider more general groups.

\smallskip

We also mention that if we do not require the irreducibility of $\Ad(\rho)$, then the prime power degree $\C$-representations $\rho$ of quasisimple groups have been classified, partially by Malle and Zalesskii \cite{Mall}, where they omit the alternating groups and their double covers, and completed by Balog, Bessenrodt, Olsson and Ono \cite{Ono} (for alternating groups) and Bessenrodt and Olsson \cite{Bess} (for their double covers). The generic examples of such representations are given by the Steinberg characters of finite groups of Lie type.

\smallskip

In general, following an outline of Michael Aschbacher, we construct an infinite family of representations of prime power degree with irreducible adjoint representations. In this case, the groups $G$ are taken to be the normalizers of certain extraspecial groups (the analogue of the Heisenberg groups over finite fields). More precisely, we establish the following:

\begin{thm} Let $p\geq 3$ be a prime, $r$ a positive integer, and $P$ an extraspecial group of order $p^{2r+1}$ and exponent $p$. Let $(\rho,V)$ be an irreducible $\C$-representation of degree $p^r$ of $P$, which is necessarily faithful. If $G$ is the normalizer in $\SL(V)$ of $P$ then $\Ad(\rho)$ is an irreducible representation of $G$.
\end{thm}

Note that this gives new cases of prime power dimensions where $G$ is not quasisimple. Moreover, due to the geometric nature of the Heisenberg groups, we remark that for some values of $p$, the groups $G$ constructed above are known to have geometric interpretations. For instance, when $p=5$ and $r=1$, Decker \cite{Dec} has proved that the group $G$ is the full symmetry group of the Horrocks-Mumford bundle constructed in \cite{HM}, which is an indecomposable rank 2 bundle on $\mathbb{P}^4$. In addition, for $p=3$ and $r=1$, as noted in \cite{Dol}, the image of $G$ in $\PGL_3$ is the group of projective automorphisms preserving the one-dimensional linear system of plane cubic curves in $\mathbb{P}^2$ given by $$t_0(x^3+y^3+z^3)+t_1xyz=0, (t_0,t_1)\in \mathbb{P}^1.$$

\smallskip

Since every finite group appears as a Galois group over a number field $F$, our construction furnishes examples of non essentially self-dual representations of the absolute Galois groups $G_F$ (for suitable number fields). Given a continuous Galois $\C$-representation $\rho$ of $G_F$, for any finite place $v$ at which $\rho$ is unramified, consider the Frobenius class $\Frob_v$ attached to $v$, and denote by $a_v$ the trace of $\rho(\Frob_v)$. The interest in $\Ad(\rho)$ stems from the following result:

\begin{prop} Let $(\rho,V)$ be a continuous irreducible $n$-dimensional complex representation of the absolute Galois group $G_{F}$ of a number field $F$. Let $\sum_{i=1}^{N}m_i\sigma_i$ be the decomposition of $\Ad(\rho)$ into irreducible components $\sigma_i$, such that $\sigma_i\not\simeq \sigma_j$ if $i\neq j$, $1\leq i,j\leq N$. Then the density $\delta(\Sigma)$ of the set $\Sigma$ of places $\{v\text{ finite}~|~a_v=0\}$ satisfies $$\delta(\Sigma)\leq 1-\frac{1}{1+\sum_{i=1}^Nm_i^2}.$$ In particular, if $\Ad(\rho)$ is irreducible then $\delta(\Sigma)\leq 1/2$.
\end{prop}

Note that this is independent of the dimension $n$ of the representation $\rho$, which contrasts well with the sharp upper bound of $1-1/{n^2}$ given by Serre in \cite{Ser}, when no condition on $\Ad(\rho)$ is imposed.

\section{Preliminaries.} An initial observation shows that if $\rho^{\vee}$ is the dual representation of $\rho$ then $\rho\otimes \rho^{\vee}$ is an $n^2$-dimensional representation that contains the trivial representation. Moreover, by Schur's Lemma, $\rho$ is irreducible if and only if $\rho\otimes \rho^{\vee}$ contains the trivial representation $\mathbf{1}$ with multiplicity one. In fact, one knows that
$\End(\rho)\cong\rho\otimes \rho^{\vee}=\textbf{1}\oplus \Ad(\rho).$

\begin{lem} Given an irreducible representation $(\rho,V)$, its adjoint $\Ad(\rho)$ is irreducible if and only if both its symmetric square $\Sym^2(\rho)$ and its alternating square $\Lambda^2(\rho)$ are irreducible.
\end{lem}

\begin{proof} As noticed above, $\rho$ is irreducible if and only if $\textbf{1}$ is not contained in $\Ad(\rho)$, so $\Ad(\rho)$ is irreducible if and only if $\textbf{1}$ is not contained in $\Ad(\Ad(\rho))$. However, $\Ad(\rho)$ is readily seen to be self-dual, since $\rho \otimes \rho^{\vee}$ and $\mathbf{1}$ are. Therefore, $$\Ad(\rho)^{\otimes 2}=\Ad(\rho)\otimes \Ad(\rho)^{\vee}=\textbf{1}\oplus \Ad(\Ad(\rho)).$$ Consider $$\sigma=(\rho\otimes \rho^{\vee})\otimes (\rho \otimes \rho^{\vee}).$$ One one hand, we get $$\sigma=(\textbf{1}\oplus \Ad(\rho))^{\otimes 2}=\textbf{1}\oplus \Ad(\rho)^{\oplus 2}\oplus \Ad(\rho)^{\otimes 2}.$$ This shows that $\Ad(\rho)$ is irreducible if and only if $\textbf{1}$ is contained in $\sigma$ with multiplicity 2, i.e., $\dim \Hom_G(\textbf{1},\sigma)=2$.

On the other hand we can write $$\sigma=(\rho\otimes \rho)\otimes (\rho\otimes \rho)^{\vee}=(\Sym^2(\rho)\oplus \Lambda^2(\rho))\otimes (\Sym^2(\rho)\oplus \Lambda^2(\rho))^{\vee},$$ which implies that $\dim \Hom_G(\textbf{1},\sigma)$ is greater than or equal to $$\dim \Hom_G(\textbf{1},\Sym^2(\rho)\otimes(\Sym^2(\rho))^{\vee})+\dim \Hom_G(\textbf{1},\Lambda^2(\rho)\otimes(\Lambda^2(\rho))^{\vee})\geq 2.$$
The equality takes place if and only if both $\Sym^2(\rho)$ and $\Lambda^2(\rho)$ are irreducible, so the conclusion follows.
\end{proof}

Recall that a representation $\rho$ is said to be essentially self-dual if $\rho^{\vee}=\rho\otimes {\chi}$, for some character $\chi$. Note than when $\rho$ is essentially self-dual there is a direct sum decomposition $$\rho\otimes \rho^{\vee}=\Sym^2(\rho)\otimes \chi \oplus \Lambda^2(\rho)\otimes \chi.$$ Since $\dim \Lambda^2(\rho)=\frac{n(n-1)}{2}$ it follows that if $n\geq 3$ and $\rho$ is essentially self-dual then $\Ad(\rho)$ is reducible. Therefore, for $n\geq 3$ the irreducibility of $\Ad(\rho)$ forces $\rho$ to be non essentially self-dual (i.e., $\rho$ is not essentially self-dual).

\subsection{Primitivity.} Recall that an irreducible representation $(\rho,V)$ of $G$ is called \emph{imprimitive} if $V$ can be written as a direct sum of $m>1$ subspaces that are permuted by $G$ transitively. We call $V$ \emph{primitive}, if it is not imprimitive. It is known that if $H$ is a proper subgroup of $G$ and $(\sigma,W)$ is a representation of $H$ then the induced representation $\Ind_H^G(\sigma)$ is imprimitive and, conversely, any imprimitive representation of $G$ is induced (see, for e.g., \cite{Karp2}).

\begin{prop} Let $(\rho,V)$ be an irreducible representation of $G$ such that $\Ad(\rho)$ is irreducible. Then $(\rho,V)$ is primitive.
\end{prop}

\begin{proof} Suppose that $V$ is imprimitive. Then $\rho=\Ind_{H}^G(\tau)$ for some proper subgroup $H$ of $G$, and a representation $\tau$ of $H$. As a result $$\rho\otimes \rho^{\vee}=\Ind_{H}^G(\tau)\otimes \Ind_{H}^G(\tau^{\vee})=\Ind_{H}^G\big(\Res_G^H(\Ind_H^G(\tau))\otimes\tau^{\vee}\big),$$ where for the second equality we have used the "push-pull" formula $\sigma_1\otimes \Ind \sigma_2=\Ind(\Res(\sigma_1)\otimes \sigma_2)$. By Mackey's Decomposition Theorem we know that $$\Res_G^H(\Ind_H^G(\tau))=\bigoplus_{s\in H\backslash G/H}\Ind_{H\cap sHs^{-1}}^H(\tau^s),$$ where $s$ runs through a set of representatives of $(H,H)$ double coset of $G$ and $\tau^s(h)$ is defined to be $\tau(s^{-1}hs)$, for $h\in H$. Hence $$\rho\otimes \rho^{\vee}=\bigoplus_{s\in H\backslash G/H}\Ind_H^{G}\big(\Ind_{H\cap sHs^{-1}}^H(\tau^s)\otimes \tau^{\vee}\big).$$
Using $s$ as the identity in the above summation, we note that the induction is reducible, so in order for $\Ad(\rho)$ to be irreducible there can be no more summands, i.e, $|H\backslash G/H|=1$. However, since $G$ is the disjoint union of $HsH$ for $s\in H\backslash G/H$, it follows that $G=H$, which contradicts the assumption that $H$ is a proper subgroup.
\end{proof}

\section{Low dimensions.}
\subsection{The case $n=2$.} Let $\mathfrak{sl}_2(\C)$ be the 3-dimensional Lie algebra of $2\times 2$ matrices with trace zero, and let
$$
\bigg\{
 \left( \begin{array}{cc}
1 & 0 \\
0 & -1 \end{array} \right) ,
 \left( \begin{array}{ccc}
0 & 1 \\
0 & 0 \end{array} \right),
 \left( \begin{array}{ccc}
0 & 0 \\
1 & 0 \end{array} \right)
\bigg\}
$$
be an ordered basis for $\mathfrak{sl}_2(\C)$. The group $\GL_2(\C)$ acts on $\mathfrak{sl}_2(\C)$ by conjugation, and the corresponding representation can be explicitly described (with respect to the chosen basis) as

\begin{displaymath}
\left( \begin{array}{cc}
a & b \\
c & d \end{array} \right) \mapsto
\text{~$\frac{1}{ad-bc}$}
\left( \begin{array}{ccc}
ad+bc & -ac & bd   \\
-2ab & a^2 &  -b^2\\
2cd & -c^2 & d^2 \end{array} \right).
\end{displaymath}
The kernel of this representation is composed only of the scalar matrices, so the representation factors through $\GL_2(\C)/\C^{\star}\cong \PGL_2(\C)$. Hence, the aforementioned representation is the 3-dimensional adjoint representation $\Ad:\GL_2(\C)\to \GL_3(\C).$ Moreover, $\GL_2(\C)$ also acts on the space of $2\times 2$ symmetric matrices, and this action with respect to the ordered basis
$$
\bigg\{
 \left( \begin{array}{cc}
0 & 1 \\
1 & 0 \end{array} \right) ,
 \left( \begin{array}{ccc}
0 & 0 \\
0 & 1 \end{array} \right),
 \left( \begin{array}{ccc}
1 & 0 \\
0 & 0 \end{array} \right)
\bigg\}
$$ is given by

\begin{displaymath}\text{$\Sym^2$:}
\left( \begin{array}{cc}
a & b \\
c & d \end{array} \right) \mapsto
\left( \begin{array}{ccc}
ad+bc & ac & bd   \\
2ab & a^2 & b^2\\
2cd & c^2 & d^2 \end{array} \right).
\end{displaymath}

\noindent Thus, in fact, $\Ad$ and $\Sym^2\otimes \det^{-1}$ are equivalent representations.

For an irreducible 2-dimensional representation $\rho:G\to \GL_2(\C)$, the projective image $\overline{G}$ is a finite subgroup of $\PGL_2(\C)\cong SO_3(\C)$. The group $SO_3(\C)$ has five finite subgroups: $C_n,D_n,A_4,S_4$ and $A_5$.

\begin{prop} Let $\rho:G\to \GL_2(\C)$ be an irreducible 2-dimensional representation of a finite group $G$, and let $\overline{G}$ be its projective image in $\PGL_2(\C)$. Then $\Ad(\rho)$ is irreducible if and only if $\overline{G}\cong A_4,S_4$ or $A_5$.
\end{prop}

\begin{proof} We rule out $C_n$ and $D_n$ since these groups do not have 3-dimensional irreducible representations. The group $A_4$ has three 1-dimensional and one 3-dimensional irreducible representations. Thus $\Ad(\rho)$ is the unique irreducible 3-dimensional representations (as the trivial representation is not contained in $\Ad(\rho)$). The irreducibility of $\Ad(\rho)$ when $\overline{G}=S_4$ follows by restricting to the normal subgroup $A_4$.

Finally, $A_5$ has one 1-dimensional, two 3-dimensional, one 4-dimensional and one 5-dimensional irreducible representations. Hence any 3-dimensional representation, and in particular $\Ad(\rho)$, must be irreducible.
\end{proof}

\subsection{The case $n=3$.} Here we make use of Hambleton and Lee's modern geometric account (\cite{HamLee}) of Blichfeldt's classification for the finite subgroups of $\PGL_3(\C)$ (\cite{Blich}). Using the primitivity condition proved in the previous section, we can restrict our attention to the primitive subgroups only. There are three primitive simple subgroups of $\PGL_3(\C)$, namely: $A_5,\PSL_2(\F_7)$ and $A_6$. The remaining three primitive subgroups of $\PGL_3(\C)$ are all solvable. We describe them in terms of the following $3\times 3$ matrices in $\GL_3(\C)$: $S=diag(1,\omega,\omega^2)$ ($\omega^3=1$), $U=diag(\epsilon,\epsilon, \epsilon\omega)$ ($\epsilon^3=\omega^2$),
\begin{displaymath}
\text{~$T$=}
 \left( \begin{array}{ccc}
0 & 1 & 0 \\
0 & 0 & 1 \\
1 & 0 & 0 \end{array} \right),
\text{~$V$=~$\frac{1}{\sqrt{-3}}$}
 \left( \begin{array}{ccc}
1 & 1 & 1 \\
1 & \omega & \omega^2 \\
1 & \omega^2 & \omega \end{array} \right).
\end{displaymath}

With this notation, the primitive, solvable subgroups of $\PGL_3(\C)$ are: $G_{216}$, the Hessian group of order 216, generated by $S,T,V,U$; $G_{72} \triangleleft G_{216}$, the normal subgroup of order 72, generated by $S,T,V,UVU^{-1}$; $G_{36} \triangleleft G_{72}$, a normal subgroup of order 36, generated by $S,T,V$.

\begin{prop} Let $\rho:G\to \GL_3(\C)$ be an irreducible 3-dimensional representation of a finite group $G$, and let $\overline{G}$ be its projective image in $\PGL_3(\C)$. Then $\Ad(\rho)$ is irreducible if and only if $\overline{G}\cong \PSL_2(\F_7),A_6, G_{72}$ or $G_{216}$.
\end{prop}

\begin{proof}
From the previous section we know that $A_5$ has no irreducible 8-dimensional representations. The group $\PSL_2(\F_7)$ has one 1-dimensional, two 3-dimensional, one 6-dimensional, one 7-dimensional and one 8-dimensional irreducible representations. The trivial representation is not contained in $\Ad(\rho)$, so $\Ad(\rho)$ is irreducible. Finally, $A_6$ has one 1-dimensional, two 5-dimensional, two 8-dimensional, one 9-dimensional, and one 10-dimensional irreducible representations. Thus $\Ad(\rho)$ is irreducible in this case.

For the following three solvable groups we use their character tables from \cite{GrimLudl}. The group $G_{36}$ has four 1-dimensional representations and two 4-dimensional irreducible representations, hence $\Ad(\rho)$ must be the sum of the two 4-dimensional irreducible representations. The character table of $G_{72}$ shows that it has four 1-dimensional representations, one 2-dimensional irreducible representation and one 8-dimensional irreducible representation. By the aforementioned arguments, $\Ad(\rho)$ is induced from either of the two 4-dimensional representations of the normal subgroup $G_{36}$ and therefore $\Ad(\rho)$ is irreducible. Likewise, $G_{216}$ has four 1-dimensional representations, one 2-dimensional representation, eight 3-dimensional representations, two 6-dimensional representations, and one 8-dimensional representation. It is not hard to see that $\Ad(\rho)$ is induced from non-normal extensions of degrees 4 and 8 and it is irreducible.
\end{proof}

\subsection{The case $n=4$.} We refer to the list of the finite primitive subgroups of $\PGL_4(\C)$ given by Blichfeldt in \cite{Blich}. The simple groups on this list are $A_5,A_6,A_7,\PSL_2(\F_7)$ and $\PSp_4(\F_3)$. We remark that all the solvable groups on the list are mapped by $\rho$ into either $\GO_4(\C)$ or $\GSp_4(\C)$, i.e., $\rho$ is either orthogonal or symplectic.

\begin{prop} Let $\rho:G\to \GL_4(\C)$ be an irreducible 4-dimensional representation of a finite group $G$, and let $\overline{G}$ be its projective image in $\PGL_4(\C)$. Then $\Ad(\rho)$ is irreducible if and only if $\overline{G}\cong A_7$ or $\PSp_4(\F_3)$.
\end{prop}

\begin{proof} We first show that if $\rho$ is either orthogonal or symplectic then $\rho$ is self-dual. Indeed, given an action of $G$ on a quadratic (or symplectic) vector space $W$ we have an action of $G$ on the dual space $W^{\vee}$ given by $(gf)(x)=f(g^{-1}x)$, for all $g\in G, f\in W^{\vee}, x\in W$. Any bilinear form $\beta:W\times W\to \C$ gives an isomorphism $\phi:W\to W^{\vee}$ mapping $x\mapsto (\phi_x:y\mapsto \beta(x,y))$, $y\in W$. Since $\beta(gx,gy)=\beta(x,y)$ it follows that $\phi$ is a $G$-equivariant isomorphism and hence the action of $G$ is self-dual. However, we know from section 2 that $\Ad(\rho)$ must be reducible in this case. Thus, $\Ad(\rho)$ cannot be irreducible if $\overline{G}$ is solvable.

The three simple groups, $A_5,A_6$ and $PSL_2(\F_7)$ have no irreducible 15-dimensional representations. The group $A_7$ has one 1-dimensional, one 6-dimensional, two 10-dimensional, two 14-dimensional, one 15-dimensional, one 21-dimensional and one 35-dimensional irreducible representations, so in this case $\Ad(\rho)$ is irreducible. Finally, as mentioned in the introduction, the adjoint of the Weil representation of $\Sp_{2n}(\F_3)$ is irreducible. In particular, if $\overline{G}\cong \PSp_4(\F_3)$ then $\Ad(\rho)$ is irreducible.
\end{proof}

\section{An infinite family for $n=p^r$.} Let $p$ be an odd prime and $r$ a positive integer. The object of this section is to construct an infinite family of groups $G$ with irreducible representations $(\rho,V)$ of degree $n=p^r$, whose adjoint $\Ad(\rho)$ is irreducible. In this case, the group $G$ will arise as the normalizer in $\SL(V)$ of a suitable extraspecial $p$-group $P$.

Recall that $P$ a finite $p$-group is called \emph{extraspecial} if its center $Z(P)$ is a cyclic group of order $p$ which coincides with the commutator group $[P,P]$, such that $P/Z(P)$ is elementary abelian. Up to isomorphism there are two nonabelian extraspecial groups of order $p^3$. One isomorphism class is represented by the Heisenberg group of $3\times 3$ upper-triangular matrices over $\F_p$ with 1's on the main diagonal. The second isomorphism class is represented by the group of maps $\Z/p^2\Z \to \Z/p^2\Z$ of the form $x\mapsto ax+b$, where $a\equiv 1$ modulo $p$, and $b\in \Z/p^2\Z$. The Heisenberg group has exponent $p$, while the second group has exponent $p^2$. We will be mainly interested in the former class of groups.

Below we collect some classical results about the extraspecial groups. Further properties of these groups can be found in \cite{Asch}.

\begin{lem} Let $P$ be an extraspecial $p$-group. Then there exists $r\geq 1$ such that $|P|=p^{2r+1}$ and $P$ is the central product of $r$ non-abelian subgroups of order $p^3$, i.e., there exist normal subgroups $N_1,\ldots, N_r$ such that (i) $P=N_1\ldots N_r$; (ii) $[N_i,N_j]=1$, whenever $i\neq j$; (iii) $N_1\ldots N_{i-1}N_{i+1}\ldots N_r \cap N_i=Z$, $\forall$ $i$.
\end{lem}

\begin{proof}
Denote by $Z$ the center $Z(P)$ of $P$, and let $z$ be a generator of this cyclic group of order $p$. Then $P/Z$ can be naturally viewed as a vector space over $\F_p$, the finite field of integers modulo $p$. We define a map $$\beta:P/Z\times P/Z \to Z$$ as follows: for any pair $x,y\in P$, if their commutator $[x,y]=x^{-1}y^{-1}xy$ equals $z^a$ (with $0\leq a \leq p-1$), we set $\beta(xZ,yZ)=a$. It is easy to check that $\beta$ is nondegenerate, bilinear and skew-symmetric, which gives $P/Z$ the structure of a nondegenerate symplectic space over $\F_p$. Then $P/Z$ has a \emph{symplectic basis} $\overline{x_1},\overline{y_1},\dots,\overline{x_r},\overline{y_r}$ such that: $\beta(\overline{x_i},\overline{y_i})=1$, $\beta(\overline{x_i},\overline{y_j})=0$ for $i\neq j$, and $\beta(\overline{x_i},\overline{x_j})=\beta(\overline{y_i},\overline{y_j})=0$ (for all $1\leq i,j \leq r$). Hence $P/Z=\oplus \overline{N_i},\text{ where $\overline{N_i}=\langle \overline{x_i},\overline{y_i}\rangle$}.$ Letting $N_i$ be the preimages of $\overline{N_i}$ in $P$ we obtain the desired characterization for $P$.
\end{proof}

\begin{lem}\emph{(\cite{Karp1})} Let $P$ be an extraspecial group of order $p^{2r+1}$. Then $P$ has exactly $p^{2r}+p-1$ inequivalent irreducible representations over $\C$, of which $p^{2r}$ are representations of degree 1 and $p-1$ are faithful representations of degree $p^r$, which are completely determined by their restriction on $Z$.
\end{lem}

\begin{lem}\emph{(\cite{Wint})} Let $P$ be an extraspecial group of order $p^{2r+1}$ and exponent $p$. If $\tilde{G}=\{\alpha\in \text{Aut}(P):\alpha|_{Z}=1\}$ and Inn$(P)$ is the group of inner automorphisms then
$\tilde{G}/\text{Inn}(P)\cong \Sp_{2r}(\F_p).$
\end{lem}

\subsection{The main construction.} We now prove Theorem 1.1, the main result of this article. Recall that $G$ is taken to be the normalizer in $\SL(V)$ of an extraspecial group $P$ with a faithful irreducible representation $(\rho,V)$ of degree $p^r$. Note that for $p=5$ and $r=1$, this group $G=P\rtimes \SL_2(\F_5)$ is used in \cite{HM} to construct an indecomposable rank 2 vector bundle on $\mathbb{P}^4$ with 15000 symmetries (the order of $G$).

\begin{proof}[Proof of Theorem 1.1.]

First we relate $G$ to the symplectic group $\text{Sp}_{2r}(\F_p)$ by the means of the following lemma:

\begin{lem} Let $E$ be the centralizer in $\SL(V)$ of $P$, then $G/EP\cong\Sp_{2r}(\F_p).$
\end{lem}

\begin{proof}
The automorphisms Aut$_{\SL(V)}(P)$ of $P$ induced by $\SL(V)$ are given by $G/E$. Let $\tilde{G}$ be the group introduced in Lemma 4.3, i.e., $\tilde{G}$ is the normal subgroup of Aut$(P)$ that acts trivially on $Z$. Then $\tilde{G}\supseteq \text{Aut}_{\SL(V)}(P)$. If $\alpha\in \tilde{G}$ then considering that the irreducible representation of dimension $p^n$ have distinct central characters (by Stone-von Neumann Theorem) it follows that $\rho\cong \rho\circ \alpha$, showing that $\alpha\in \text{Aut}_{\SL(V)}(P)$. Consequently, $\tilde{G}=\text{Aut}_{\SL(V)}(P)$ and Lemma 4.3 assures that $G/EP\cong \tilde{G}/\text{~Inn}(P)\cong \Sp_{2r}(\F_p).$
\end{proof}

We know from Lemma 4.1 that $P/Z$ can be viewed as a $2r$-dimensional symplectic space $W$ over $\F_p$. The corresponding symplectic form $\beta$ gives an isomorphism $\phi:W\to W^{\vee}$ between $W$ and its dual space $W^{\vee}$, sending $$x\mapsto (\phi_x:y\mapsto \beta(x,y)).$$ Since $\Sp_{2r}(\F_p)$ preserves $\beta$, so does $G$ (by Lemma 4.4), which implies that $\phi$ is a $G$-equivariant isomorphism. Furthermore, the group $\Sp_{2r}(\F_p)$ is transitive on $W\setminus \{0\}$, while the dual group $W^{\vee}$ can be identified with the group of characters $\Hom(W,\C)$ of $W$. Therefore, $G$ acts transitively on the set $\Hom^{\star}(W,\C)$ of nontrivial characters of $W$.

Given $x\in W$, consider a lift $\tilde{x}$ to $P$ and put $x\cdot M=\tilde{x}M\tilde{x}^{-1}$, which is independent of the choice of the lift. Since $W$ is an elementary abelian group, we can decompose $\Ad(\rho)$ according to the eigenspaces corresponding to the characters of $W$. For $\chi\in \Hom(W,\C)$ consider the corresponding weight space for $W$ on $\Ad(\rho)$: $$A_{\chi}=\{M\in \Ad(\rho)~:~\tr{M}=0 \text{ and }x\cdot M=\chi(x)M,\forall ~x\in W\}.$$ If $\Lambda=\{\chi\in \Hom(W,\C): A_{\chi}\neq 0\}$ is the set of weights then $\Ad(\rho)$ decomposes as $\Ad(\rho)=\bigoplus_{\chi\in \Lambda} A_{\chi}.$ The group $G$ permutes the weight spaces by $gA_{\chi}=A_{g\chi}$, $g\in G$, $\chi\in \Lambda$. Since $\Ad(\rho)$ does not contain the trivial representation, it follows that $\Lambda\subseteq \Hom^{\star}(W,\C)$. Moreover, since $G$ is transitive on $\Hom^{\star}(W,\C)$ we obtain that $\Lambda=\Hom^{\star}(W,\C)$. The dimensions of the $p^{2r}-1$ weight spaces add up to the dimension of $\Ad(\rho)$, which is $p^{2r}-1$. Therefore, each space $A_{\chi}$ is of dimension 1.

If $W'\subset \Ad(\rho)$ is a $G$-invariant subspace then $W'$ restricted to $W$ is contained in $\oplus_{\chi\in \Lambda}A_{\chi}$, and since each summand is 1-dimensional it follows that $A_{\chi}\subset W'$ for some $\chi$. However, in that case $$\Ad(\rho)=G\cdot A_{\chi}\subset G\cdot W'=W',$$ which is a contradiction. Consequently, $G$ is irreducible on $\Ad(\rho)$.

\end{proof}

\begin{rem} When $p=3$ and $r=1$ the image $\overline{G}$ in $\PGL_3(\C)$ of the group $G$ constructed above is the Hessian group $G_{216}$ found in Proposition 3.2. Indeed, the lemma implies that $G/EP\cong \Sp_2(\F_3)=\SL_2(\F_3).$ Since $|\SL_2(\F_3)|=24$ it follows that $|G|=24|EP|$, and $27=|P|$ divides $EP$. Consequently, 81 divides the order of $G$ and thus 27 divides the order of $\overline{G}$. From the four possible choices only $G_{216}$ satisfies this property.
\end{rem}

\section{An application.} We start with a brief discussion on the Artin $L$-functions before proving Proposition 1.2.

Let $F$ be a number field and let $\rho:G_{F}\to GL(V)$ be a continuous representation of the absolute Galois group $G_{F}=\Gal(\overline{F}/F)$ on a $n$-dimensional $\C$-vector space $V$. The continuity implies that $\rho$ has finite image, so it factors through the projection $\Gal(\overline{F}/F) \to \Gal(K/F)$, for some finite extension $K/F$. The Artin L-function attached to $\rho$ is the Euler product of local factors $L(\rho,s)=\prod_{v} L_v(\rho,s),$ defined as $$L_v(\rho,s)=\det([1-\rho(\Frob_v)q_v^{-s}]|_{V^{I_v}})^{-1},$$ where $\Frob_v\in \Gal(\overline{F_v}/F_v)$ is a Frobenius element at a place $v$, $V^{I_v}$ is the subspace of $V$ fixed by the inertia group $I_v$ at $v$, and $q_v$ is the order of the residue field of $F_v$.

Let $S$ be the finite set of places outside which $\rho$ is unramfied. Then for $v\notin S$, since the Frobenius conjugacy class $\{\rho(\Frob_v)\}$ is semisimple, the eigenvalues $\alpha_{1,v},\dots,\alpha_{n,v}$ of the corresponding linear transformation are roots of unity. Accordingly, the local factor can be written as $$L_v(\rho,s)=\prod_{j=1}^n(1-\alpha_{j,v}q_v^{-s})^{-1}.$$
Denote by $a_v$ the trace of the Frobenius at $v$, i.e., $a_v=\sum_{j=1}^n \alpha_{j,v}.$
Notice that $$\log L_v(\rho,s)=\frac{a_v}{q_v^s}+\sum_{m\geq 2}\frac{\sum_{j=1}^n \alpha_{j,v}^m}{q_v^{ms}}.$$
Using the fact that $|\alpha_{j,v}|=1$ we obtain for real $s>1$ $$\sum_{m\geq 2}\frac{|\sum_{j=1}^n \alpha_{j,v}^m|}{q_v^{ms}}\leq n\sum_{m\geq 2} \frac{1}{q_v^{ms}},$$ which implies that $$\log L(\rho,s) \sim \sum_{v\notin \Sigma}\frac{a_v}{q_v^s} \text{ as } s\to 1^{+},$$ where the relation $\sim$ means that the two sides agree up to a function of $s$ which is $o(\log(\frac{1}{s-1}))$, and $\Sigma=\{v\text{ finite }|~a_v=0\}$.

\begin{proof}[Proof of Proposition 1.2] Consider the Artin L-function $L(\rho,s)$ attached to a continuous irreducible $n$-dimensional representation $(\rho,V)$ of the absolute Galois group $G_{F}$, and let $\Sigma$ as defined above.

The order of the pole at $s=1$ of $L(\rho,s)$ is the multiplicity of the trivial representation in $\rho$. Therefore, if $\rho$ is irreducible then $-ord_{s=1}L(\rho\otimes \rho^{\vee},s)=1$.
Since $\rho\otimes \rho^{\vee}=\textbf{1}\oplus \Ad(\rho)$ we have
$$\rho\otimes \rho^{\vee}\otimes (\rho\otimes \rho^{\vee})^{\vee}=\textbf{1}\oplus2\Ad(\rho)\oplus\Ad(\rho)^{\otimes 2}.$$ The key observation is that the trivial representation is contained in $\Ad(\rho)^{\otimes 2}$ with multiplicity $\sum_{i=1}^Nm_i^2$, therefore $$-ord_{s=1}L(\rho\otimes \rho^{\vee}\otimes \rho \otimes \rho^{\vee},s)=1+\sum_{i=1}^Nm_i^2.$$ Since $\rho$ is a $\C$-representation of a finite group, the dual representation $\rho^{\vee}$ is isomorphic to the complex conjugate representation $\overline{\rho}$, and so the Frobenius trace of $v$ on $\rho\otimes \rho^{\vee}$ is given by $a_v\cdot \overline{a}_v=|a_v|^2$. Therefore we get
$$\log L(\rho\otimes \rho^{\vee},s)\sim \sum_{p\notin \Sigma}\frac{|a_v|^2}{q_v^s}\sim\log\Big(\frac{1}{s-1}\Big)$$ and similarly $$\log L(\rho\otimes \rho^{\vee}\otimes (\rho\otimes \rho^{\vee})^{\vee},s)\sim \sum_{p\notin \Sigma}\frac{|a_v|^4}{q_v^s}\sim \Big(1+\sum_{i=1}^Nm_i^2\Big)\log\Big(\frac{1}{s-1}\Big).$$ Consider the sequence $\{b_v\}$, with $b_v=1$ if $v\notin \Sigma$ and $b_v=0$, otherwise. By the Cauchy-Schwarz inequality, it follows that $$\sum_{v\notin \Sigma}\frac{|a_v|^2}{q_v^s}=\sum_{v}\frac{|a_v^2b_v|}{q_v^s}\leq \Bigg(\sum_v\frac{|a_v|^4}{q_v^s}\Bigg)^{1/2}\Bigg(\sum_v \frac{b_v}{q_v^s}\Bigg)^{1/2}.$$ If $\delta(\Sigma)$ is the density of the set $\Sigma$ then by construction $$\sum_v \frac{b_v}{q_v^s}\sim (1-\delta(\Sigma))\log\Big(\frac{1}{s-1}\Big)$$ and so the previous inequality reads as $1\leq (1+\sum_{i=1}^Nm_i^2)(1-\delta(\Sigma))$, yielding $$\delta(\Sigma)\leq 1-\frac{1}{1+\sum_{i=1}^Nm_i^2}.$$ As a consequence, it follows that when $\Ad$ is irreducible (i.e., $N=1$ and $m_1=1$) the traces of Frobenius classes in finite Galois groups are nonzero for at least half of the primes.
\end{proof}

\subsection{Some examples}
As previously seen in section 2, the irreducibility of $\Ad(\rho)$ forced $\rho$ to be non essentially self-dual in dimension $n\geq 3$. Since the density estimate in Proposition 1.2 depends only on the multiplicity of the irreducible constituents of $\Ad(\rho)$, one could also consider self-dual representations $\rho$ for which $\Ad(\rho)$ has only two inequivalent simple components (and thus $\delta(\Sigma)\leq 2/3$ by Proposition 1.2). We provide two examples of such representations.

\begin{ex}\emph{(}based on \emph{\cite{DR})} Take $G=\SL_2(\F_5)$ and let $\pi:G\to \GL_2(\C)$ be an irreducible 2-dimensional representation. By Proposition 3.1, since the projective image of $G$ in $\PGL_2(\C)$ is $A_5$, we infer that $\Ad(\pi)$ is an irreducible 3-dimensional representation. If $\rho=\Ad(\pi)$ then using the fact that $\Ad(\pi)=\Sym^2(\pi)\otimes \det(\pi)^{-1}$ it follows that $$\Sym^2(\rho)=\Sym^4(\pi)\otimes \det(\pi)^{-2}\oplus \textbf{\emph{1}}.$$ Now, $\Ad(\rho)$ is the sum of $\rho$ and $\Sym^4(\pi)\otimes \det(\pi)^{-2}$, which are both irreducible.
\end{ex}

\begin{ex} Take $G=\SL_2(\F_9)$ (the double cover of $A_6$) and let $\rho:G\to \GL_4(\C)$ be an irreducible 4-dimensional representation. Using \emph{\cite{MM}} (Proposition 3.1) we find that $\Sym^2(\rho)$ is an irreducible 10-dimensional representation. Furthermore, $\Lambda^2(\rho)$ contains the trivial representation with multiplicity one, so we can write $\Lambda^2(\rho)=\emph{\textbf{1}}\oplus \tau$, where $\tau$ is a 5-dimensional representation that factors through $A_6$. By the character table of $\SL_2(\F_9)$ it has one 1-dimensional, two 4-dimensional, two 5-dimensional, four 8-dimensional, one 9-dimensional and three 10-dimensional irreducible representations. Thus, $\tau$ must be irreducible and $\Ad(\rho)$ is the sum of $\tau$ and $\Sym^2(\rho)$.
\end{ex}

\smallskip

After completing this manuscript, the author learned of a paper of Guralnick and Tiep \cite{Gur}, which contains Theorem 1.1, albeit arranged in a different way. Our focus is on the application (Proposition 1.2) to lacunarity.

\bigskip

\centerline{\Large{\textbf{Acknowledgements}}}

\smallskip

\noindent We would like to thank Dinakar Ramakrishnan for suggesting the problem and for his invaluable help and encouragement. We are grateful to Michael Aschbacher for outlining the construction that appears in Section 4 and also for his useful comments and suggestions. We would like to acknowledge helpful conversations with Iurie Boreico, Hadi Hedayatzadeh and Andrei Jorza.

\bigskip

\noindent \sc{Liubomir Chiriac}\\
\sc{California Institute of Technology},\\ 
Department of Mathematics, MC 253-37,\\ 
Pasadena, CA 91125\\
lchiriac@caltech.edu

\end{document}